\documentclass[11pt]{article}

\usepackage[english]{babel}
\usepackage[T1]{fontenc}
\usepackage[utf8]{inputenc}

\usepackage{amsmath,amssymb,amsthm,amsfonts}

\usepackage{amsthm}
\usepackage{aliascnt}
\theoremstyle{plain}

\newtheorem{theorem}{Theorem}[section]

\newaliascnt{lemma}{theorem}
\newtheorem{lemma}[lemma]{Lemma}
\aliascntresetthe{lemma}

\newaliascnt{proposition}{theorem}
\newtheorem{proposition}[proposition]{Proposition}
\aliascntresetthe{proposition}

\newaliascnt{corollary}{theorem}

\aliascntresetthe{corollary}

\newaliascnt{conjecture}{theorem}
\newtheorem{conjecture}[conjecture]{Conjecture}
\aliascntresetthe{conjecture}

\newaliascnt{property}{theorem}

\aliascntresetthe{property}

\theoremstyle{definition}

\newaliascnt{definition}{theorem}
\newtheorem{definition}[definition]{Definition}
\aliascntresetthe{definition}

\newaliascnt{assumption}{theorem}

\aliascntresetthe{assumption}

\newaliascnt{example}{theorem}
\newtheorem{example}[example]{Example}
\aliascntresetthe{example}

\newaliascnt{exercise}{theorem}

\aliascntresetthe{exercise}

\theoremstyle{remark}

\newaliascnt{remark}{theorem}
\newtheorem{remark}[remark]{Remark}
\aliascntresetthe{remark}

\usepackage{mathtools}
\usepackage{mathrsfs}
\usepackage{centernot}
\usepackage{blkarray}
\usepackage[scr=boondox]{mathalfa}

\usepackage{graphicx}
\usepackage{tikz}
\usetikzlibrary{
    arrows.meta,
    arrows,
    matrix,
    positioning,
    shapes,
    calc
}
\usepackage{forest}
\usepackage{stackengine}
\stackMath

\usepackage{booktabs}
\usepackage{array}

\usepackage{float}
\usepackage[font=footnotesize,labelfont=bf]{caption}
\usepackage{subcaption}
\usepackage{tikz}
\usetikzlibrary{arrows.meta,positioning}

\usepackage[colorlinks=true,
            linkcolor=blue,
            citecolor=blue,
            urlcolor=blue]{hyperref}
\usepackage[capitalise,nameinlink]{cleveref}
\crefname{theorem}{Theorem}{Theorems}
\Crefname{theorem}{Theorem}{Theorems}

\crefname{lemma}{Lemma}{Lemmas}
\Crefname{lemma}{Lemma}{Lemmas}

\crefname{proposition}{Proposition}{Propositions}
\Crefname{proposition}{Proposition}{Propositions}

\crefname{corollary}{Corollary}{Corollaries}
\Crefname{corollary}{Corollary}{Corollaries}

\crefname{conjecture}{Conjecture}{Conjectures}
\Crefname{conjecture}{Conjecture}{Conjectures}

\crefname{property}{Property}{Properties}
\Crefname{property}{Property}{Properties}

\crefname{definition}{Definition}{Definitions}
\Crefname{definition}{Definition}{Definitions}

\crefname{assumption}{Assumption}{Assumptions}
\Crefname{assumption}{Assumption}{Assumptions}

\crefname{example}{Example}{Examples}
\Crefname{example}{Example}{Examples}

\crefname{exercise}{Exercise}{Exercises}
\Crefname{exercise}{Exercise}{Exercises}

\crefname{remark}{Remark}{Remarks}
\Crefname{remark}{Remark}{Remarks}
\usepackage{datetime2}
\usepackage{filecontents}
\usepackage{matlab-prettifier}

\usepackage[
    colorinlistoftodos,
    bordercolor=orange,
    backgroundcolor=orange!20,
    linecolor=orange,
    textsize=small
]{todonotes}

\usepackage{lineno}
\modulolinenumbers[5]
\usepackage{graphicx}

\newcommand{\new}[1]{{\em #1}}

\DeclareMathAlphabet{\mathbbold}{U}{bbold}{m}{n}

\newcommand{\zeror}{-\infty} 
\newcommand{\mv}{m} 

\newcommand{\R}{\mathbb R}

\newcommand{\rmax}{\mathbb{R}_{\max}}

\newcommand{\cycle}{c}

\newcommand{\permutation}{p}

\usepackage{amssymb}

\renewcommand{\geq}{\geqslant}
\renewcommand{\leq}{\leqslant}

\renewcommand{\le}{\leq}
\renewcommand{\ge}{\geq}

\newcommand{\N}{\mathbb{N}}

\DeclareMathOperator{\uval}{ldeg}

\begin{document}

\title{Properties of the Tropical Characteristic Polynomial of Symmetric Matrices}

\author{
Dariush Kiani\thanks{Department of Mathematics and Computer Science, Amirkabir University of Technology (Tehran Polytechnic), Tehran, Iran. Email: \texttt{dkiani@aut.ac.ir}}
\and
Hanieh Tavakolipour\thanks{Department of Mathematics and Computer Science, Amirkabir University of Technology (Tehran Polytechnic), Tehran, Iran. Email: \texttt{h.tavakolipour@aut.ac.ir}}
}

\date{}
\maketitle

\begin{abstract}
We investigate the combinatorial structure of the tropical characteristic polynomial of symmetric matrices using the tropical permanents of their principal submatrices. We establish new inequalities for the leading coefficients of the tropical characteristic polynomial, revealing concavity properties of the coefficient sequence and yielding necessary conditions for a sequence to arise as the coefficient sequence of the tropical characteristic polynomial of a symmetric matrix. These results provide a deeper understanding of the structure of tropical characteristic polynomials associated with symmetric matrices.
\end{abstract}


\medskip


\section{Introduction}
Tropical algebra, $\mathbb{R}_{\max} := (\mathbb{R} \cup \{-\infty\}, \oplus, \otimes)$, is a semiring defined by replacing the classical operations with their tropical counterparts: tropical addition is given by the maximum, while tropical multiplication corresponds to ordinary addition. This framework has found numerous applications across areas such as combinatorics, optimization, and algebraic geometry (see, e.g., \cite{baccelli1992synchronization,butkovivc2010max,viro2001dequantization,itenberg2009tropical,maclagan2015introduction}). More recently, it has also been applied in numerical analysis and numerical linear algebra, particularly in the approximation of roots of polynomials and eigenvalues of matrices over the field of Puiseux series (see \cite{hook2019max,van2018polynomial,hook2018max,hook2017incomplete,noferini2015tropical, akian2025factorization}).

The approximation of eigenvalues via tropical algebra is generally based on the computation of tropical eigenvalues. Once the tropical characteristic polynomial of $A \in \rmax^{n \times n}$
 has been determined, the tropical eigenvalues can be computed in linear time \cite{graham1972efficient}. 
 
 To compute the tropical eigenvalues, it is generally unnecessary to determine all the terms of the tropical characteristic polynomial. To the best of the authors' knowledge, the fastest algorithm for computing the required terms of the tropical characteristic polynomial of an $n \times n$ matrix has complexity $O(n^3)$. This algorithm is based on parametric optimal assignment techniques~\cite{gassner2010fast}. On the other hand, no polynomial-time algorithm is known for computing all the terms of the tropical characteristic polynomial of $A$ in general (see, e.g.,~\cite{butkovivc2007job}).
 
The coefficients of the terms of the tropical characteristic polynomial have a direct combinatorial interpretation: they are closely related to maximum-weight permutations in the weighted directed graph associated with a matrix. The structural properties of a matrix strongly influence these maximum-weight permutations.
In certain structured settings, combinatorial methods significantly reduce computational complexity.
For instance, the computational complexity of computing the tropical eigenvalues can 
be reduced to polynomial time when considering special classes of matrices, 
such as symmetric matrices over $\{0,-\infty\}$, pyramidal matrices, Monge and Hankel matrices, tridiagonal Toeplitz and pentadiagonal Toeplitz matrices (see \cite{butkovivc2007job}, \cite{tavakolipour2020asymptotics}, \cite{tavakolipour2018tropical}).
Also for positive definite matrices defined in the context of symmetrized tropical semiring, the cost of eigenvalue approximation can decrease to linear complexity (see \cite{akian2025spectral}). For further structural insights into tropical positive definite matrices, we refer the reader to \cite{al2024tropicalizing,yu2015tropicalizing}. 

Unlike the structured matrix classes discussed above, symmetric matrices over $\rmax$ need not be sparse, that is, they need not contain many $-\infty$ entries, and generally possess no structural assumptions beyond the symmetry of their entries like positive definite matrices. Consequently, the computation of the tropical characteristic polynomial remains challenging even under the symmetry assumption. The main objective of this paper is therefore to investigate the properties of the tropical characteristic polynomial of symmetric matrices.
 

Let $A$ be a symmetric matrix in $\rmax^{n \times n}$. 
The main objective of this paper is to investigate the structural properties of the sequence $(\delta_m)_{0\leq m\leq n}$, where $\delta_m$ denotes the coefficient of the term of degree $n-m$ in the tropical characteristic polynomial of a symmetric matrix,
\[
P_A(x)=(\delta_0 \otimes  x^{\otimes n})\oplus (\delta_1 \otimes  x^{\otimes n-1}) \oplus \cdots \oplus \delta_n,
\]
where
$
x^{\otimes k}
=
\underbrace{x\otimes\cdots\otimes x}_{k\text{ times}}
$.

We show that the coefficients $\delta_m, m \leq 4$ satisfy some concavity restrictions. These inequalities provide the first structural restrictions on the coefficients of tropical characteristic polynomials of symmetric matrices.

The paper is organized as follows. In \Cref{sec:2}, we begin with a review of the basic principles of tropical polynomials, matrices, graphs, and tropical eigenvalues. In \Cref{sec:3}, we present several characterizations of maximum permutations over symmetric matrices. In \Cref{sec:4}, we establish new inequalities for the initial coefficients of the tropical characteristic polynomial of a symmetric matrix over $\rmax$, yielding necessary conditions on the coefficient sequences of such polynomials. In \Cref{sec:5}, we illustrate the results of \Cref{sec:4} through the truncated Newton polygon of the tropical characteristic polynomial of a symmetric matrix over $\rmax$, thereby providing a geometric interpretation of the obtained inequalities. We conclude in \Cref{sec:6} with directions for future research, including a conjectural extension of our coefficient inequalities to all coefficients of the tropical characteristic polynomial.


  \section{Preliminaries}\label{sec:2}

We work over the tropical (max-plus) semiring
$
\rmax=\mathbb{R}\cup\{-\infty\},
$
equipped with tropical addition and multiplication
\[
a\oplus b=\max\{a,b\},\qquad
a\otimes b=a+b.
\]
The additive identity is $\varepsilon=-\infty$, and the multiplicative identity is $0$. For a positive integer $k$, the tropical power of $a$ is
\[
a^{\otimes k}
=\underbrace{a\otimes\cdots\otimes a}_{k\text{ times}}
=ka,
\]
with the convention that $a^{\otimes 0}=0$.

The operations $\oplus$ and $\otimes$ extend naturally to matrices and vectors. For matrices
$A=(a_{ij})$ and $B=(b_{ij})$ of compatible dimensions,
\[
(A\oplus B)_{ij}=a_{ij}\oplus b_{ij},
\]
and
\[
(A\otimes B)_{ij}
=
\bigoplus_{k}a_{ik}\otimes b_{kj}
=
\max_k(a_{ik}+b_{kj}).
\]

Throughout the paper, $\rmax^{n\times n}$ denotes the set of all $n\times n$ matrices over $\rmax$. We denote by $I$ the tropical identity matrix, whose diagonal entries are $0$ and whose off-diagonal entries are $\varepsilon$. 

For a detailed discussion on the definition of formal polynomials, polynomial functions, their factorization, and tropical roots (including multiplicities over 
$\rmax$), we refer the reader to \cite{akian2025factorization}. This section provides a brief review of the necessary concepts.

A \new{formal polynomial} $P$ over $\rmax$ is a sequence $(P_k)_{k\in\mathbb N}$, where $\N $ is the set of natural numbers (including $0$), with $P_k\in\rmax$ such that $P_k=\zeror$ for all but finitely many $k$.
 The \textit{degree} and \textit{lower degree} of $P$ are respectively defined by
\begin{equation}\label{deg}\deg(P):=\sup\{k \in \N \mid P_k \neq \zeror\},
  \;
  \uval (P) := \inf\{k \in \N\;|\;P_k \neq \zeror\}.
\end{equation}
When $P = \zeror$, we have $\deg(P)= -\infty$ and $\uval(P) = +\infty$.
To any formal polynomial $P$ over $\rmax$, with degree $n$ and lower degree $\mv$,
we associate a \new{polynomial function} 
\begin{equation}\label{widehat_p}\widehat{P}: \rmax \rightarrow \rmax, \;x \mapsto \widehat{P}(x)= \bigoplus_{k=\mv}^{n}P_{k}\otimes x^{\otimes k}=
\max_{m\leq k \leq n}(P_k+kx)\enspace .
\end{equation}

One crucial difference between tropical polynomials and formal polynomials over  $\rmax$ is that we can find polynomials $P$ and $Q$ with different coefficients such that
$\hat P(x)=\hat Q(x)$ for all $x \in \rmax$. 

\begin{definition}[Tropical roots] \label{def_corners}
Given a tropical formal polynomial $P$ over $\rmax$,
 the non-$\zeror$ \new{tropical roots} are the points at which the maximum 
in the \cref{widehat_p} of the associated polynomial function as a supremum of monomial functions,
is attained at least twice (i.e.\ by at least two different monomials). If $P$ has no constant term, then $P$ is said to have a tropical root at $\zeror$.
\end{definition}
The non-zero corners of $P$ are equivalently the points of non-differentiability of the associated polynomial function $\widehat{P}$ restricted to $\R$.

\begin{definition}[Multiplicity of a tropical root]
The \textit{multiplicity} of a tropical root $r\neq \zeror$ of the formal polynomial $P$ is the difference between the largest and the smallest exponent of the monomials of $P$ which attain the maximum at $r$. If $P$ has a tropical root at $\zeror$, the multiplicity of $\zeror$ is the lower degree $\uval{P}$ of $P$, defined in \cref{deg}.
\end{definition}
In ordinary notation, the multiplicity of a corner $r\neq \zeror$ is equivalent to the change of slope of the graph of the associated polynomial function $\widehat{P}$ at $c$.
\begin{definition}[Newton polygon]
Let $P$ be a formal polynomial over $\rmax$, with degree $n$ and lower degree $\mv$,
The \emph{Newton polygon} of $P$ is the upper boundary of the convex hull of the set
\[
\{(k,P_{k})\mid m\le k\le n,\; P_{k}\neq -\infty\}.
\]
\end{definition}

See \Cref{fig_c}, for an example of a Newton polygon. A point $(k,P_{k})$ is called a \emph{saturated index} if it lies on the Newton polygon, and a \emph{non-saturated index} otherwise. It is well known that the indices $0$ and $n$ are always saturated.

The following standard result, which follows from the properties of the Legendre-Fenchel transform~\cite{rockafellar1970convex},
establishes the relationship between tropical eigenvalues and the Newton polygon.
\begin{proposition}[See for instance~\protect{\cite[Prop.~2.4]{akian2016non}}]\label{root_trop}
Let $P$ be a formal polynomial. The tropical roots of $P$ coincide with the opposite of the slopes of the Newton polygon of $P$.
The multiplicity of a root $r$ of $P$ coincides with the horizontal length of the segment  of the Newton polygon
with slope $-r$.
\end{proposition}

 For a positive integer $n$, we denote by $[n] := \{1,\ldots,n\}$. 
A \new{cycle} of length $k$ in $[n]$ is a sequence 
$\cycle_k=(i_{1}i_{2}\ldots i_{k})$ of distinct elements of $[n]$, 
with the convention that $i_{k+1}=i_1$. 
Every permutation of $[n]$ can be decomposed uniquely into disjoint cycles.

 Given $A =(a_{ij}) \in \rmax^{n \times n}$, we denote by $D_A$ the weighted digraph with the node set $[n]$ and arc set $E=\{(i,j)\;:\; a_{ij}>-\infty\}$ and $w(i,j)=a_{ij}$ for all $(i,j)\in E$.
\begin{definition}[Weight of permutations and cycles in $D_A$]
For any permutation $\permutation$ of $[n]$ in $D_A$, the weight of $\permutation$ 
is given by
\[ w(\permutation)=\sum_{i \in[n]}a_{i\permutation(i)}\enspace ,\]
and the weight of a cycle $\cycle_k=(i_{1}i_{2}\ldots  i_{k})$ of length $k$ in $D_A$ is given by
\[w(\cycle_k)=\sum_{\ell\in [k]} a_{i_\ell i_{\ell+1}}\enspace,\]
where $i_{k+1}=i_1$.
\end{definition}

\begin{definition}[Maximum cycle mean]  Given $A\in \rmax^{n \times n}$, the symbol $\lambda(A)$ will stand for the \new{maximum cycle mean} of $A$, that is:
\[\lambda(A)=\max_{c} \frac{w(c)}{l(c)},\]
where the maximization is taken over all cycles in $D_A$. Since the number of cycles in $D_A$ is finite, the maximum is well defined.
If $D_A$ does not have any cycle, then $\lambda(A)=-\infty$. \end{definition}
\begin{lemma}[\protect{\cite{butkovivc2010max}}]\label{max_cycle_mean}
Let $A\in \rmax^{n \times n}$. Then we have
\[\lambda(A)=\max_{k=1, \ldots, n}\frac{\delta_k}{k}.\]

\end{lemma}

 Let $S_n$ denote the set of all permutations of $[n]$. A principal submatrix of
$A=(a_{ij})\in\rmax^{n\times n}$ is obtained by selecting the same subset of rows and columns. For indices
$
1\le i_1<i_2<\cdots<i_k\le n,
$
the corresponding principal submatrix is denoted by
$
A(i_1,i_2,\ldots,i_k)$.
The \new{tropical permanent} of a square matrix
$A=(a_{ij})\in\rmax^{n\times n}$ is defined by
\[
\operatorname{maper}(A)
=
\bigoplus_{\sigma\in S_n}
\bigotimes_{i=1}^{n}
a_{i,\sigma(i)}
=
\bigoplus_{\sigma\in S_n} w(\sigma)
=\max_{\sigma\in S_n}w(\sigma)
=\max_{\sigma\in S_n}\sum_{i=1}^{n}a_{i,\sigma(i)}.
\]
Equivalently, $\operatorname{maper}(A)$ is the optimal value of the classical assignment problem with cost matrix $A$. 
The \new{tropical characteristic polynomial function} of a matrix
$A\in\rmax^{n\times n}$ is defined by
\[
\chi_A(x)
=
\operatorname{maper}(A\oplus x\otimes I)
=
\operatorname{maper}
\begin{pmatrix}
a_{11}\oplus x & a_{12} & \cdots & a_{1n}\\
a_{21} & a_{22}\oplus x & \cdots & a_{2n}\\
\vdots & \vdots & \ddots & \vdots\\
a_{n1} & a_{n2} & \cdots & a_{nn}\oplus x
\end{pmatrix}.
\]
It follows directly from the definition that
\[
\chi_A(x)
=
(\delta_{0}\otimes x^{\otimes n}) \oplus (\delta_{1}\otimes x^{\otimes n-1}) \oplus \cdots \oplus (\delta_{n-1}\otimes  x) \oplus \delta_n\enspace,
\]
where $\delta_0=0$.
In ordinary notation, the tropical characteristic polynomial can be written as
\[
\chi_A(x)
=
\max\{nx+\delta_0,\,(n-1)x+\delta_1,\ldots,x+\delta_{n-1},\,\delta_n\}.
\]
Hence, $\chi_A(x)$ is the upper envelope of finitely many affine functions. Consequently, it is a convex piecewise-linear function.
The coefficients $\delta_k$, $k\in[n]$, are purely combinatorial quantities arising from the tropical permanents of principal submatrices of order $k$. Consequently, in the tropical characteristic  polynomial, the coefficient of the term of degree $n-k$ is indexed by $k$. The following theorem establishes this correspondence precisely.

\begin{theorem}\label{delt_char}
for $k\in [n]$,
\[
\delta_k
=
\max_{B\in\mathcal P_{\,k}(A)}
\operatorname{maper}(B),
\]
where $\mathcal P_{\,k}(A)$ denotes the collection of all principal submatrices of order $k$.
\end{theorem}

A term
$
\delta_k\otimes x^{\otimes(n-k)}
$
is called \emph{inessential} if
\[
\delta_k\otimes x^{\otimes n-k}
\le
\bigoplus_{i\neq k}
\delta_i\otimes x^{\otimes n-i},
\]
for every $x\in\mathbb{R}$.
Otherwise, it is called \emph{essential}. Equivalently, if the term
$\delta_k\otimes x^{\otimes n-k}$
is inessential, then
\[
\chi_A(x)
=
\bigoplus_{i\neq k}
\delta_i\otimes x^{\otimes n-i},
\]
for all $x\in\mathbb{R}$.
Therefore, inessential terms do not affect the tropical characteristic polynomial as a function and may be omitted. The inessential terms of the tropical characteristic polynomial coincide with the non-saturated indices of its Newton polygon. 
  
\begin{example}
Let
\[
A =
\begin{pmatrix}
0 & 2 & -\infty & 1 \\
1 & 0 & 3 & -\infty \\
-\infty & 2 & 0 & 4 \\
1 & -\infty & 2 & 0
\end{pmatrix}
\in \mathbb{R}_{\max}^{4\times 4}.
\]
We compute the coefficients of $\chi_A(x)$ step by step using \Cref{delt_char}.

\textbf{$k=0$ (principal submatrices of order $0$).}
By convention,
\[
\delta_0=0.
\]

\textbf{$k=1$ (principal submatrices of order $1$).}
These are the diagonal entries, so
\[
\delta_1=\max\{0,0,0,0\}=0.
\]

\textbf{$k=2$ (principal submatrices of order $2$).}
We consider all $2\times2$ principal submatrices:
\[
\delta_2=\max\{\operatorname{maper}(B):B\in\mathcal{P}_2(A)\}.
\]
The maximum is
$
\delta_2=6
$,
which is attained by the principal submatrix $B=A(3,4)$ with permutation $\sigma=(34)$.

\textbf{$k=3$ (principal submatrices of order $3$).}
We remove one row and the corresponding column and compute the maximum assignment value over the resulting $3\times3$ principal submatrices. The maximum is
$
\delta_3=6
$,
which is attained by the principal submatrix $B=A(2,3,4)$ with permutation $\sigma=(2)(34)$.

\textbf{$k=4$ (principal submatrices of order $4$).}
Since the only principal submatrix of order $4$ is $A$ itself,
$
\delta_4=\operatorname{maper}(A)=10
$
with $\sigma = (1234)$.
\medskip

Hence,
\[
\chi_A(x)
=
(0 \otimes x^{\otimes 4})
\oplus (0 \otimes x^{\otimes 3})
\oplus (6 \otimes x^{\otimes 2})
\oplus 6 \otimes x
\oplus 10.
\]
which is 
\[
\chi_A(x)
=
 x^{\otimes 4}
\oplus  x^{\otimes 3}
\oplus (6 \otimes x^{\otimes 2})
\oplus 6 \otimes x
\oplus 10.
\]
\medskip

\end{example}


 \begin{definition}[Tropical eigenvalue] \label{algebraic}Let $A \in \rmax^{ n \times n}$. The \new{tropical eigenvalues} of $A$, denoted by $\mu_{1}(A)\geq \cdots\geq \mu_{n}(A)$, are the tropical roots of its characteristic polynomial. These can be computed as described in  \Cref{root_trop}.
\end{definition}
\begin{remark}[\cite{butkovivc2010max}]\label{greatest}
Let $A \in \rmax^{ n \times n}$ with tropical eigenvalues $\mu_{1}(A)\geq \cdots\geq \mu_{n}(A)$. Then $\mu_1(A)=\lambda(A)$.
\end{remark}
In \Cref{ex_poly1}, we compute the tropical eigenvalues of a matrix, including their multiplicities, utilizing the Newton polygon method  in  \Cref{root_trop}.

 \begin{example}\label{ex_poly1}
Let us consider the following symmetric matrix of size $7 \times 7$:

\[A =\begin{pmatrix}
 -20 &   -1&   -13&     5&   -13&    -4&     0\\
     -1 &  -10  &  12   & -4   &  3  &  -8   & -6\\
    -13   & 12  & -14 &   17 &   -4  &   2  & -12\\
     5  &  -4   & 17 &  -20 &   12  &  15&     8\\
   -13  &   3  &  -4   & 12  &  -2  &  11  &  -9\\
   -4 &   -8  &   2  &  15   & 11 &    2  & -16\\
     0   & -6 &  -12&     8 &   -9   &-16 &  -10
\end{pmatrix}.\]

  We computed the characteristic polynomial as follows:
  \[P_A(x)=     x^{\otimes7}   \oplus  (2\otimes x^{\otimes6})   \oplus     (34 \otimes x^{\otimes5})   \oplus    (38 \otimes x^{\otimes4})   \oplus    (56\otimes x^{\otimes3})   \oplus     (62\otimes x^{\otimes2})   \oplus     (62\otimes x)   \oplus     62.\]

 Figure \ref{fig_c} displays the Newton polygon of $P_A$. Using the Newton polygon of $P_A$, the lengths of the segments with slopes $0$, $-6$, $-11$, and $-17$ are $2$, $1$, $2$, and $2$, respectively. Consequently, the tropical roots of $P_A$ and so tropical eigenvalues of $A$ are $\mu_1=17$, $\mu_2= 11$, $\mu_3=6$ and $\mu_4=0$ with respective multiplicities $2$, $2$, $1$ and $2$.

This example illustrates how the Newton polygon provides both the tropical eigenvalues and their multiplicities in a geometric way.

\begin{figure}

          \includegraphics[width=10cm]{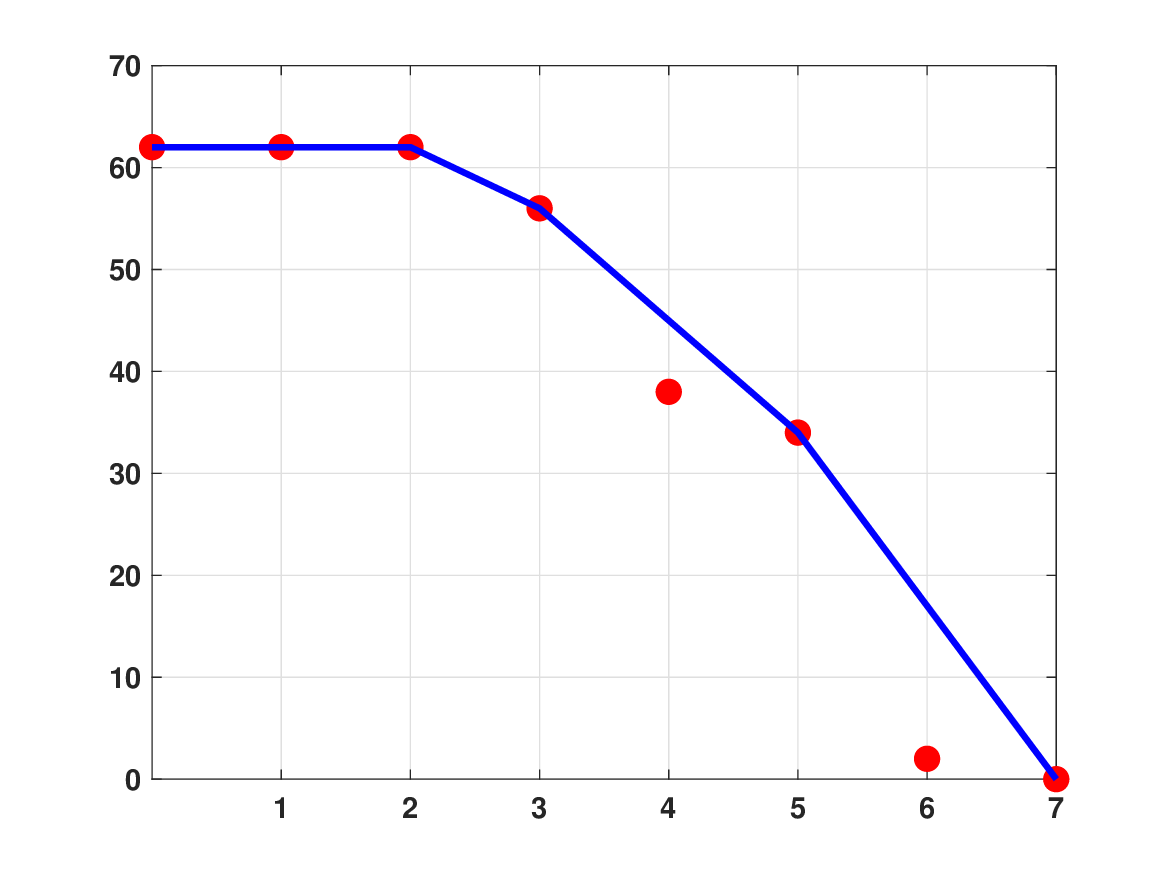}

\caption{Configuration of points of Newton polygon associated with the characteristic polynomial of the matrix $A$  and its tropical eigenvalues in \Cref{ex_poly1}.}\label{fig_c}
\end{figure}

\end{example}

\section{Properties of permutation weights in $D_A$ for symmetric matrices}\label{sec:3}

In this section, we establish properties concerning the weights of cycles in $D_A$ for a symmetric matrix $A$. In particular, we focus on the weights of even-length and odd-length cycles, deriving results that will be used in the next section. In the following lemma we prove that every even cycle is dominated by a collection of disjoint transpositions.

 \begin{lemma}\label{even}
Let $A\in \rmax^{n\times n}$ be a symmetric matrix. If
\[
c_{2k}=(i_1 i_2 \ldots i_{2k})
\]
is a cycle of even length $2k$ in $D_A$, where $k\in\N$ and $k \le \frac{n}{2}$, then there exists a permutation consisting of $k$ disjoint $2$-cycles whose weight is greater than or equal to the weight of $c_{2k}$.
\end{lemma}

\begin{proof}
See \Cref{fig:even} for an illustration of the construction. Let $A$ be a symmetric matrix and let
\[
c_{2k}=(i_1 i_2 \ldots i_{2k})
\]
be an even cycle in $D_A$.
By definition,
\begin{equation}\label{c2k}
w(c_{2k})
=
a_{i_1 i_2}+a_{i_2 i_3}+\cdots+a_{i_{2k-1}i_{2k}}+a_{i_{2k}i_1}.
\end{equation}

Consider the permutations
\[
p_{2k}
=
(i_1 i_2)(i_3 i_4)\cdots(i_{2k-1} i_{2k})
\]
and
\[
\tilde p_{2k}
=
(i_2 i_3)(i_4 i_5)\cdots(i_{2k} i_1),
\]
which consist of $k$ disjoint $2$-cycles.

Since $A$ is symmetric, we have
\begin{align}
w(p_{2k})
&=
(a_{i_1 i_2}+a_{i_2 i_1})
+\cdots+
(a_{i_{2k-1}i_{2k}}+a_{i_{2k}i_{2k-1}})
\nonumber\\
&=
2a_{i_1 i_2}
+2a_{i_3 i_4}
+\cdots+
2a_{i_{2k-1}i_{2k}},
\label{p2k}
\end{align}
and similarly,
\begin{align}
w(\tilde p_{2k})
&=
(a_{i_2 i_3}+a_{i_3 i_2})
+\cdots+
(a_{i_{2k}i_1}+a_{i_1 i_{2k}})
\nonumber\\
&=
2a_{i_2 i_3}
+2a_{i_4 i_5}
+\cdots+
2a_{i_{2k}i_1}.
\label{ptilde2k}
\end{align}

Combining \Cref{c2k,p2k,ptilde2k}, we obtain
\[
w(c_{2k})
=
\frac{w(p_{2k})+w(\tilde p_{2k})}{2}.
\]

Hence, at least one of the two numbers
$w(p_{2k})$ and $w(\tilde p_{2k})$
is greater than or equal to $w(c_{2k})$.
Therefore, there exists a permutation consisting of $k$ disjoint $2$-cycles whose weight is greater than or equal to the weight of $c_{2k}$.
\end{proof}

\begin{figure}[H]
\centering

\begin{subfigure}{\textwidth}
\centering
\begin{tikzpicture}[
    ->,
    >=Stealth,
    shorten >=2pt,
    auto,
    node distance=2.5cm,
    thick,
    main node/.style={circle,draw,font=\sffamily\small\bfseries}
]

\node[main node] (1) {1};
\node[main node] (2) [right=of 1] {2};
\node[main node] (3) [right=of 2] {3};
\node[main node] (4) [right=of 3] {4};

\path
(1) edge[bend left] node[below]{\tiny 5} (2)
(2) edge[bend left] node[below]{\tiny 3} (3)
(3) edge[bend left] node[below]{\tiny 9} (4)
(4) edge[bend left] node[below]{\tiny $-2$} (1);

\end{tikzpicture}
\caption{The even cycle $c_{2k}=(1\,2\,3\,4)$ with weight
$w(c_{2k})=5+3+9-2=15$.}
\end{subfigure}

\vspace{0.8cm}

\begin{subfigure}{\textwidth}
\centering
\begin{tikzpicture}[
    ->,
    >=Stealth,
    shorten >=2pt,
    auto,
    node distance=2.5cm,
    thick,
    main node/.style={circle,draw,font=\sffamily\small\bfseries}
]

\node[main node] (1) {1};
\node[main node] (2) [right=of 1] {2};
\node[main node] (3) [right=of 2] {3};
\node[main node] (4) [right=of 3] {4};

\path
(1) edge[bend left] node[below]{\tiny 5} (2)
(2) edge[bend left] node[below]{\tiny 5} (1)
(3) edge[bend left] node[below]{\tiny 9} (4)
(4) edge[bend left] node[below]{\tiny 9} (3);

\end{tikzpicture}
\caption{The permutation $p_{2k}=(12)(34)$ of $\{1,2,3,4\}$ with weight
$w(p_{2k})=28$.}
\end{subfigure}

\vspace{0.8cm}

\begin{subfigure}{\textwidth}
\centering
\begin{tikzpicture}[
    ->,
    >=Stealth,
    shorten >=2pt,
    auto,
    node distance=2.5cm,
    thick,
    main node/.style={circle,draw,font=\sffamily\small\bfseries}
]

\node[main node] (1) {1};
\node[main node] (2) [right=of 1] {2};
\node[main node] (3) [right=of 2] {3};
\node[main node] (4) [right=of 3] {4};

\path
(1) edge[bend left] node[below]{\tiny $-2$} (4)
(2) edge[bend left] node[below]{\tiny 3} (3)
(3) edge[bend left] node[below]{\tiny 3} (2)
(4) edge[bend left] node[below]{\tiny $-2$} (1);

\end{tikzpicture}
\caption{The permutation $\tilde p_{2k}=(14)(23)$ of $\{1,2,3,4\}$ with weight
$w(\tilde p_{2k})=2$.}
\end{subfigure}

\caption{
Illustration of \Cref{even}. Subfigure~(A) shows the even cycle
$c_{2k}=(1\,2\,3\,4)$. Subfigures~(B) and~(C) show the permutations
$p_{2k}=(12)(34)$ and $\tilde p_{2k}=(14)(23)$, respectively.
In this example,
$
w(c_{2k})
=
\frac{w(p_{2k})+w(\tilde p_{2k})}{2}.
$
}
\label{fig:even}
\end{figure}

\begin{remark} By \Cref{even}, if $A\in \rmax^{n\times n}$ is symmetric and $k\in\mathbb{N}$ with $k\le \frac{n}{2}$, then every cycle of length $2k$  in $D_A$ can be replaced by a product of $k$ disjoint $2$-cycles  in $D_A$ of at least the same weight. \end{remark}
 
 \begin{lemma}\label{odd}
Let $A\in \rmax^{n \times n}$ be a symmetric matrix, and let
$
c_k=(i_1 i_2 \ldots i_k)
$
be a cycle of length $k\leq n$ in $D_A$.
Then
\[
2w(c_k)\leq k\delta_2,
\]
where $\delta_2$ denotes the maximum weight of a $2$-cycle in $D_A$.
\end{lemma}

\begin{proof}
Consider the $k$ cycles of length two
\[
(i_1 i_2),\,
(i_2 i_3),\,
\ldots,\,
(i_{k-1} i_k),\,
(i_k i_1).
\]
Since $\delta_2$ is the maximum weight of a $2$-cycle in $D_A$, we have
\[
w(i_\ell i_{\ell+1})\leq \delta_2
\qquad
\text{for all } \ell\in[k],
\]
where $i_{k+1}=i_1$.
Because $A$ is symmetric,
\[
w(i_\ell i_{\ell+1})
=
a_{i_\ell i_{\ell+1}}
+
a_{i_{\ell+1}i_\ell}
=
2a_{i_\ell i_{\ell+1}}.
\]
Therefore,
\[
\sum_{\ell=1}^k
w(i_\ell i_{\ell+1})
=
2\sum_{\ell=1}^k
a_{i_\ell i_{\ell+1}}
=
2w(c_k).
\]
Hence,
\[
2w(c_k)
=
\sum_{\ell=1}^k
w(i_\ell i_{\ell+1})
\leq
k\delta_2.
\]

See \Cref{fig:odd} for an illustration.
\end{proof}

\begin{figure}[!ht]
\centering

\begin{subfigure}{0.6\textwidth}
\centering
\begin{tikzpicture}[
->,
>=Stealth,
shorten >=2pt,
auto,
node distance=2.5cm,
thick,
main node/.style={circle,draw,font=\sffamily\small\bfseries}
]

\node[main node] (1) {1};
\node[main node] (2) [right=of 1] {2};
\node[main node] (3) [right=of 2] {3};

\path
(1) edge[bend left] node[above]{3} (2)
(2) edge[bend left] node[above]{2} (3)
(3) edge[bend left] node[below]{5} (1);

\end{tikzpicture}
\caption{The cycle $c_3=(1\,2\,3)$.}
\end{subfigure}

\vspace{0.8cm}

\begin{subfigure}{0.3\textwidth}
\centering
\begin{tikzpicture}[
->,
>=Stealth,
shorten >=2pt,
auto,
node distance=2.5cm,
thick,
main node/.style={circle,draw,font=\sffamily\small\bfseries}
]

\node[main node] (1) {1};
\node[main node] (2) [right=of 1] {2};

\path
(1) edge[bend left] node[above]{3} (2)
(2) edge[bend left] node[below]{3} (1);

\end{tikzpicture}
\caption{$(1\,2)$.}
\end{subfigure}
\hfill
\begin{subfigure}{0.3\textwidth}
\centering
\begin{tikzpicture}[
->,
>=Stealth,
shorten >=2pt,
auto,
node distance=2.5cm,
thick,
main node/.style={circle,draw,font=\sffamily\small\bfseries}
]

\node[main node] (2) {2};
\node[main node] (3) [right=of 2] {3};

\path
(2) edge[bend left] node[above]{2} (3)
(3) edge[bend left] node[below]{2} (2);

\end{tikzpicture}
\caption{$(2\,3)$.}
\end{subfigure}
\hfill
\begin{subfigure}{0.3\textwidth}
\centering
\begin{tikzpicture}[
->,
>=Stealth,
shorten >=2pt,
auto,
node distance=2.5cm,
thick,
main node/.style={circle,draw,font=\sffamily\small\bfseries}
]

\node[main node] (3) {3};
\node[main node] (1) [right=of 3] {1};

\path
(3) edge[bend left] node[above]{5} (1)
(1) edge[bend left] node[below]{5} (3);

\end{tikzpicture}
\caption{$(3\,1)$.}
\end{subfigure}

\caption{
Illustration of the proof of \Cref{odd}. Subfigure (A) shows the cycle
$c_3=(1\,2\,3)$.
Subfigures (B)--(D) show the associated $2$-cycles
$(1\,2)$, $(2\,3)$, and $(3\,1)$, respectively.
}
\label{fig:odd}
\end{figure}

\begin{lemma}\label{lem_new}
Let $A=(a_{ij})\in \rmax^{n\times n}$ be a symmetric matrix.
Then, for every $i\in[n]$,
\[
\frac{\delta_i}{i}
\leq
\max\left\{\delta_1,\frac{\delta_2}{2}\right\}.
\]
In particular, at least one of the following inequalities holds:
\[
\frac{\delta_i}{i}\leq \delta_1,
\qquad\text{or}\qquad
\frac{\delta_i}{i}\leq \frac{\delta_2}{2}.
\]
\end{lemma}

\begin{proof}
We consider two cases.
\\
\textbf{Case 1:}
\[
\max_{i,j} a_{ij}=a_{kk},
\qquad
k\in[n].
\]
In this case,
$
\delta_1=a_{kk}$. 
Since $a_{kk}$ is the maximal entry of $A$, the mean weight of every cycle is less than or equal to $a_{kk}$. Hence,
$
\frac{\delta_i}{i}\leq \delta_1
$.

\smallskip

\noindent
\textbf{Case 2:}
\[
\max_{i,j} a_{ij}=a_{pq},
\qquad
p,q\in[n],
\quad
p\neq q.
\]
Since $A$ is symmetric, we have
$
a_{pq}=a_{qp}
$
and therefore the $2$-cycle $(pq)$ has weight
$
\delta_2= a_{pq}+a_{qp}=2a_{pq}
$.
Thus,
$
\frac{\delta_2}{2}=a_{pq}
$.
Since $a_{pq}$ is the maximal entry of $A$, the mean weight of every cycle is less than or equal to $a_{pq}$. Hence,
\[
\frac{\delta_i}{i}\leq \frac{\delta_2}{2}.
\]
\end{proof}

\begin{theorem}
Let $A=(a_{ij})\in \rmax^{n\times n}$ be a symmetric matrix. Then the largest tropical eigenvalue of $A$ is given by
$\max_{i,j\in[n]} a_{ij}.
$
\end{theorem}
\begin{proof}
By \Cref{greatest} together with \Cref{max_cycle_mean}, we have
\[
\mu_1(A)=\lambda_{\max}(A)
=
\max_{i\in[n]}\frac{\delta_i}{i}.
\]
So using \Cref{lem_new},
\[
\lambda_{\max}(A)
=
\max\left\{\delta_1,\frac{\delta_2}{2}\right\}.
\]
Since
$
\delta_1=\max_{i\in[n]}a_{ii}
$,
and $A$ is symmetric,
\[
\frac{a_{ij}+a_{ji}}2=a_{ij},
\]
we obtain
\[
\max_{i,j\in[n]}a_{ij}
\le
\max\left\{\delta_1,\frac{\delta_2}{2}\right\}.
\]
On the other hand,
\begin{align*}
\frac{\delta_2}{2}
&=
\max\left\{
\max_{i\neq j}\frac{a_{ij}+a_{ji}}2,
\max_{i\neq j}\frac{a_{ii}+a_{jj}}2
\right\} \\
&=
\max\left\{
\max_{i\neq j}a_{ij},
\max_{i\neq j}\frac{a_{ii}+a_{jj}}2
\right\} \\
&\le
\max_{i,j\in[n]}a_{ij}.
\end{align*}
Hence
\[
\max\left\{\delta_1,\frac{\delta_2}{2}\right\}
=
\max_{i,j\in[n]}a_{ij},
\]
and therefore
\[
\mu_1(A)=\lambda_{\max}(A)
=
\max_{i,j\in[n]}a_{ij}.
\]
\end{proof}

\section{Coefficient inequalities for the cases $m=3$ and $m=4$}\label{sec:4}

In this section, assuming that $A\in \rmax^{n\times n}$ is symmetric, we prove that for $m=3$ and $m=4$, at least one of the following inequalities holds, revealing a concavity property among the leading terms of the sequence $(\delta_m)$:
\begin{enumerate}
    \item $\delta_m+\delta_{m-2}\le 2\delta_{m-1}$;
    \item $\delta_m+2\delta_{m-3}\le 3\delta_{m-2}$.
\end{enumerate}


\begin{theorem}\label{prop1}
Let $A$ be an $n\times n$ symmetric matrix. Then at least one of the following inequalities holds:
\begin{enumerate}
\item
\[
\delta_3+\delta_1\le 2\delta_2,
\]
\item
\[
\delta_3+2\delta_0\le 3\delta_1.
\]
\end{enumerate}
\end{theorem}

\begin{proof}
Assume that the second inequality does not hold. Since $\delta_0=0$, this implies
\begin{equation}\label{size3_eq}
\delta_1<\frac{\delta_3}{3}.
\end{equation}

By \Cref{lem_new}, we have
\[
\frac{\delta_3}{3}\le \frac{\delta_2}{2}.
\]
Combining this with \cref{size3_eq}, we obtain
\[
\delta_3+\delta_1
<\delta_3+\frac{\delta_3}{3}
=\frac{4\delta_3}{3}
\le 2\delta_2.
\]
Hence the first inequality holds.
\end{proof}

\begin{theorem}\label{prop}
Let $A$ be an $n\times n$ symmetric matrix. Then at least one of the following inequalities holds:
\begin{enumerate}
\item
\[
\delta_4+\delta_2\le 2\delta_3,
\]
\item
\[
\delta_4+2\delta_1\le 3\delta_2.
\]
\end{enumerate}
\end{theorem}

\begin{proof}
We distinguish cases according to the cycle decomposition of a permutation attaining $\delta_4$.

By \Cref{even}, every even cycle can be decomposed into disjoint $2$-cycles of at least the same weight. Hence the only possible cycle decompositions for a permutation attaining $\delta_4$ are:
\begin{enumerate}
\item two disjoint $2$-cycles;
\item one $3$-cycle together with a loop.
\end{enumerate}

Throughout the proof, we use the notation introduced in \Cref{notations}. The logical structure of the argument is summarized in \Cref{sketch}.

\begin{table}[H]
\begin{tabular}{ll}
Notation& Definition   \\
\hline
 $p_k^{(q)}$& Permutation of a subset of $[k]$ attaining $\delta_q$     \\
$\tilde{p}_k^{(q)}$& Permutation of a subset of $[k]$ attaining $\delta_q$     \\
$c_k^{(q)}$& Cycle of length $k$ appearing in a permutation attaining $\delta_q$    \\
$N(\delta_q)$& Node set corresponding to $\delta_q$\\
$N(p_k^{(q)})$& Node set corresponding to $p_k^{(q)}$\\
$N(c_k^{(q)})$& Node set corresponding to $c_k^{(q)}$\\
$w(p_k^{(q)})$& Weight of $p_k^{(q)}$\\
$w(c_k^{(q)})$& Weight of $c_k^{(q)}$\\
\hline
\end{tabular}\caption{Used notations  in the proof of \Cref{prop}}\label{notations}
\end{table}
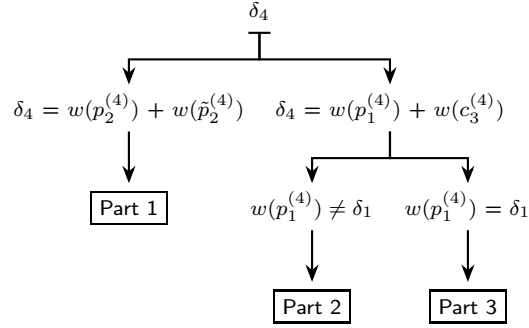
\begin{figure}[!h]
\begin{center}
\scriptsize
\begin{forest}
  for tree={
    align=center,
    parent anchor=south,
    child anchor=north,
    font=\sffamily,
    edge={thick, -{Stealth[]}},
    l sep+=10pt,
    edge path={
      \noexpand\path [draw, \forestoption{edge}] (!u.parent anchor) -- +(0,-10pt) -| (.child anchor)\forestoption{edge label};
    },
    if level=0{
      inner xsep=0pt, tikz={\draw [thick] (.south east) -- (.south west);}
    }{}
  }
  [
  $\delta_4$
    [$\delta_{4}\mbox{~=~}w(p_2^{(4)})\mbox{~+~}w(\tilde{p}_2^{(4)})$[\boxed{\mbox{Part 1}}]][$\delta_{4}\mbox{~=~}w(p_1^{(4)})\mbox{~+~}w(c_3^{(4)})$[$w(p_1^{(4)})\neq\delta_1$
    [\boxed{\mbox{Part 2}}]]
    [$w(p_1^{(4)})\mbox{~=~}\delta_1$[\boxed{\mbox{Part 3}}]]
  ]]]]
\end{forest}
\end{center}
\caption{Illustration of different parts of the proof of \Cref{prop}}
\label{sketch}
\end{figure}

\textbf{Case 1.}
Assume that
\[
\delta_4=w(p_2^{(4)})+w(\tilde p_2^{(4)}).
\]
Assume that the second inequality fails to hold, so we have 
\begin{equation}\label{eq8} 3\delta_2<\delta_4+2\delta_1 \enspace.
\end{equation}
Without loss of generality, assume that $N(\delta_1)\cap N(p_2^{(4)})=\emptyset$. Therefore, we have 
\begin{equation}\label{eq_6}
\delta_1+w(p_2^{(4)}) \leq \delta_3.
\end{equation}
Also we have $2w(\tilde{p}_2^{(4)})\leq 2\delta_2$. So 
\[w(\tilde{p}_2^{(4)}) + \delta_2 \leq 3\delta_2 - w(\tilde{p}_2^{(4)})\]
So, after adding $w(p_2^{(4)})$ to both sides we have:
\begin{equation}\label{eq7}
w(p_2^{(4)})+w(\tilde{p}_2^{(4)}) + \delta_2 \leq w(p_2^{(4)})+3\delta_2 - w(\tilde{p}_2^{(4)})\enspace,
\end{equation}
where the left-hand side of \cref{eq7} is equal to $\delta_4+\delta_2$ and the right hand side equals $2w(p_2^{(4)})+3\delta_2-\delta_4$ since
\[w(p_2^{(4)})+3\delta_2 - w(\tilde{p}_2^{(4)})=2w(p_2^{(4)})+3\delta_2-\delta_4\enspace.\]
Therefore, by \cref{eq7} we have 
\begin{equation}
\delta_4+\delta_2 \leq 2w(p_2^{(4)})+3\delta_2-\delta_4\enspace.
\end{equation}
Therefore, 
\begin{eqnarray}
\delta_4+\delta_2 &\leq&3\delta_2 -\delta_4 +2w(p_2^{(4)})\nonumber\\
&<&(2\delta_1+\delta_4)-\delta_4+2w(p_2^{(4)})\label{eq9}\\
&=&2(\delta_1+w(p_2^{(4)}))\nonumber\\
&\leq&2\delta_3\enspace,\label{eq10}
\end{eqnarray}
where the second and the last inequalities are by \cref{eq8} and \cref{eq_6}, respectively.
\\
\textbf{Case 2.}
Assume that
\[
\delta_4=w(p_1^{(4)})+w(c_3^{(4)})
\]
with $w(p_1^{(4)})\neq\delta_1$.
So we have 
\begin{equation}\label{eq2}
\delta_1+w(p_1^{(4)}) \leq \delta_2\enspace.
\end{equation}
Assume that the first inequality fails to hold, i.e.,
\[
\delta_4+\delta_2>2\delta_3.
\]
We show that this implies the second inequality. So we have
\[w(p_1^{(4)})+w(c_3^{(4)}) + \delta_2 > 2 \delta_3\]
\[ w(p_1^{(4)})+ \delta_2>2 \delta_3-w(c_3^{(4)})=\delta_3+(\delta_3-w(c_3^{(4)}))\geq \delta_3\geq w(c_3^{(4)})\]
Therefore, we have
\begin{equation}\label{eq3}
w(p_1^{(4)})+ \delta_2>w(c_3^{(4)})
\end{equation}
So,
\begin{eqnarray}
\delta_4+2\delta_1&=&w(p_1^{(4)})+w(c_3^{(4)})+2\delta_1\nonumber\\
&\leq&w(p_1^{(4)})+w(p_1^{(4)})+ \delta_2+2\delta_1\label{eq4}\\
&=& \delta_2 +2(w(p_1^{(4)})+\delta_1)\nonumber\\
&\leq&\delta_2 +2\delta_2 =3\delta_2\label{eq1}\enspace,
\end{eqnarray}

where \ref{eq4} and \ref{eq1} are by \ref{eq3} and \ref{eq2}, respectively.\\

\textbf{Case 3.}
Assume that
\[
\delta_4=w(p_1^{(4)})+w(c_3^{(4)})
\]
with $w(p_1^{(4)})=\delta_1$.

Assume that the second inequality fails, i.e.,
\[
\delta_2 < \frac{1}{3}(\delta_4+2\delta_1).
\]
We show that then the first inequality must hold.
We have
\begin{eqnarray}
\delta_4+\delta_2&<& \frac{1}{3}(4\delta_4+2\delta_1) \nonumber\\
&=&\frac{1}{3}(4w(p_1^{(4)})+4w(c_3^{(4)})+2\delta_1)\nonumber\\
&=&\frac{1}{3}(6\delta_1+4w(c_3^{(4)}))\nonumber\\
&=&\frac{2}{3}(3\delta_1+2w(c_3^{(4)}))\nonumber\\
&\leq&\frac{2}{3}(3\delta_3)=2\delta_3.\nonumber\enspace
\end{eqnarray}

Indeed, if
\[
c_3^{(4)}=(ijk),
\]
then by \Cref{odd},
\[
2w(c_3^{(4)})=w((ij))+w((jk))+w((ki)).
\]
Since these $2$-cycles do not involve the node corresponding to $\delta_1$, the quantity
\[
3\delta_1+2w(c_3^{(4)})
\]
is realized by three permutations of length three. Therefore,
\[
3\delta_1+2w(c_3^{(4)})\le 3\delta_3.
\]

\end{proof}

\section{Possible and impossible saturation patterns for symmetric matrices}\label{sec:5}

The inequalities established in the previous section admit a natural interpretation in terms of the Newton polygon associated with the tropical characteristic polynomial of a symmetric matrix. In particular, they imply that between any two consecutive saturated indices of the Newton polygon, there can be at most one non-saturated index.

To illustrate this phenomenon, we consider the truncated Newton polygon corresponding to the tropical characteristic polynomial of symmetric matrices in $\rmax^{n\times n}$, where $n\ge 4$. Since the inequalities involve only the first five coefficients, the truncated polygon is completely determined by $\delta_0,\delta_1,\delta_2,\delta_3,\delta_4$. The remaining coefficients $\delta_i$, for $i\ge 5$, are omitted from the figures.

We emphasize that this represents only a local portion of the Newton polygon; in general, higher coefficients may affect the global shape of the polygon and the set of saturated indices.
The only relevant feature of these diagrams is the pattern of saturated and non-saturated indices; the exact coordinates of the points play no role in our analysis. See \Cref{fig:truncated_newton}.

From these figures, we observe that symmetry imposes strong restrictions on the possible configurations. In particular, the patterns shown in \Cref{fig:impossible} cannot occur for symmetric matrices, which reflects the concavity-type constraints established in the previous section.


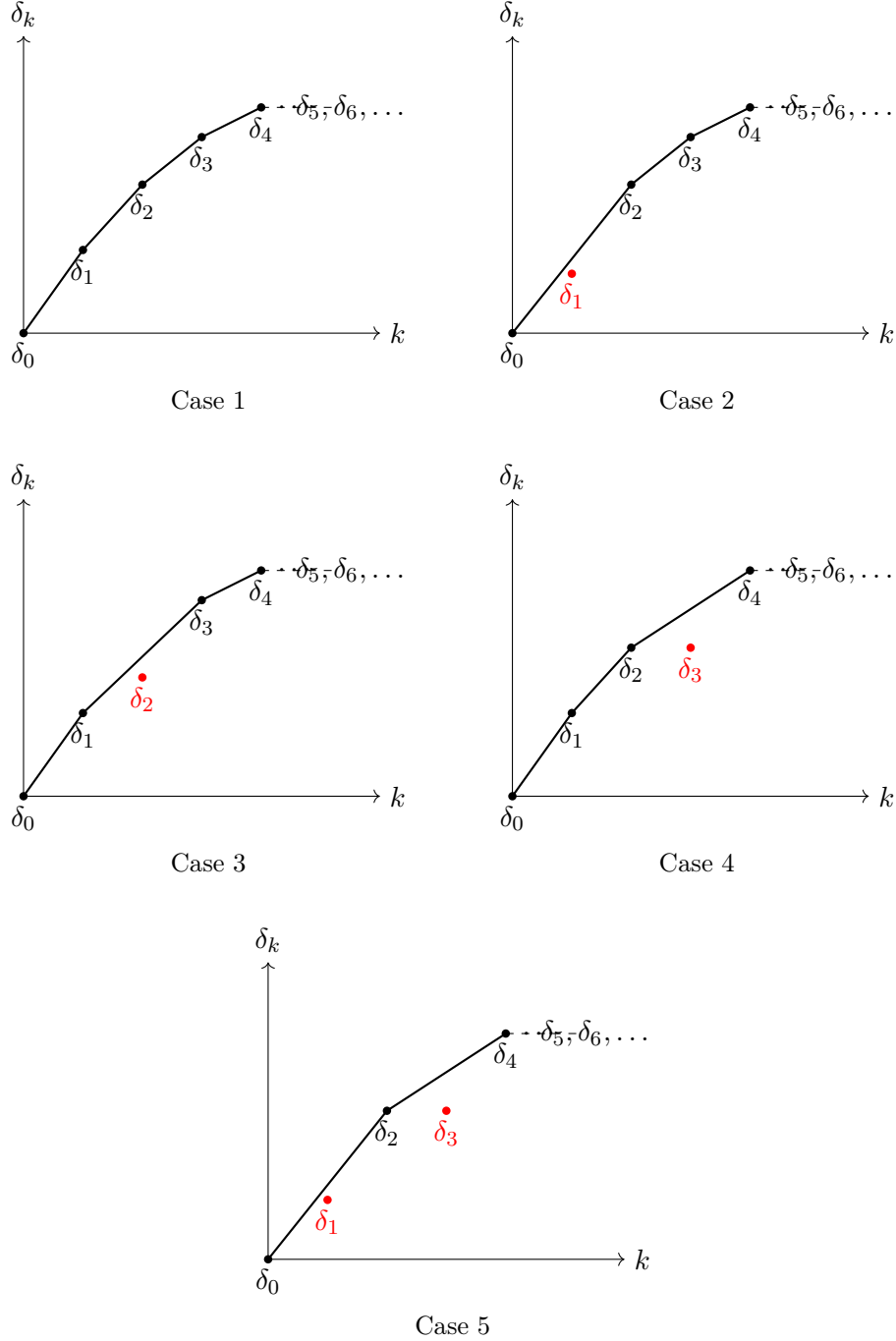
\begin{figure}[H]
\centering

\begin{minipage}{0.48\textwidth}
\centering
\begin{tikzpicture}[scale=0.8]

\draw[->] (0,0) -- (6,0) node[right] {$k$};
\draw[->] (0,0) -- (0,5) node[above] {$\delta_k$};

\fill (0,0)   circle (2pt) node[below] {$\delta_0$};
\fill (1,1.4) circle (2pt) node[below] {$\delta_1$};
\fill (2,2.5) circle (2pt) node[below] {$\delta_2$};
\fill (3,3.3) circle (2pt) node[below] {$\delta_3$};
\fill (4,3.8) circle (2pt) node[below] {$\delta_4$};

\draw[thick]
(0,0) -- (1,1.4) -- (2,2.5) -- (3,3.3) -- (4,3.8);

\node at (4.6,3.8) {$\cdots$};
\draw[dashed] (4,3.8) -- (5.2,3.8);

\node at (5.5,3.8) {$\delta_5,\delta_6,\ldots$};

\end{tikzpicture}
\small Case 1
\end{minipage}
\hfill
\begin{minipage}{0.48\textwidth}
\centering
\begin{tikzpicture}[scale=0.8]

\draw[->] (0,0) -- (6,0) node[right] {$k$};
\draw[->] (0,0) -- (0,5) node[above] {$\delta_k$};

\fill (0,0)   circle (2pt) node[below] {$\delta_0$};
\fill[red] (1,1) circle (2pt) node[below] {$\delta_1$};
\fill (2,2.5) circle (2pt) node[below] {$\delta_2$};
\fill (3,3.3) circle (2pt) node[below] {$\delta_3$};
\fill (4,3.8) circle (2pt) node[below] {$\delta_4$};

\draw[thick]
(0,0)  -- (2,2.5) -- (3,3.3) -- (4,3.8);

\node at (4.6,3.8) {$\cdots$};
\draw[dashed] (4,3.8) -- (5.2,3.8);

\node at (5.5,3.8) {$\delta_5,\delta_6,\ldots$};

\end{tikzpicture}

\small Case 2
\end{minipage}

\vspace{0.6cm}

\begin{minipage}{0.48\textwidth}
\centering
\begin{tikzpicture}[scale=0.8]

\draw[->] (0,0) -- (6,0) node[right] {$k$};
\draw[->] (0,0) -- (0,5) node[above] {$\delta_k$};

\fill (0,0)   circle (2pt) node[below] {$\delta_0$};
\fill (1,1.4) circle (2pt) node[below] {$\delta_1$};
\fill[red] (2,2) circle (2pt) node[below] {$\delta_2$};
\fill (3,3.3) circle (2pt) node[below] {$\delta_3$};
\fill (4,3.8) circle (2pt) node[below] {$\delta_4$};

\draw[thick]
(0,0) -- (1,1.4) -- (3,3.3) -- (4,3.8);

\node at (4.6,3.8) {$\cdots$};
\draw[dashed] (4,3.8) -- (5.2,3.8);

\node at (5.5,3.8) {$\delta_5,\delta_6,\ldots$};

\end{tikzpicture}
\small Case 3
\end{minipage}
\hfill
\begin{minipage}{0.48\textwidth}
\centering
\begin{tikzpicture}[scale=0.8]

\draw[->] (0,0) -- (6,0) node[right] {$k$};
\draw[->] (0,0) -- (0,5) node[above] {$\delta_k$};

\fill (0,0)   circle (2pt) node[below] {$\delta_0$};
\fill (1,1.4) circle (2pt) node[below] {$\delta_1$};
\fill (2,2.5) circle (2pt) node[below] {$\delta_2$};
\fill[red] (3,2.5) circle (2pt) node[below] {$\delta_3$};
\fill (4,3.8) circle (2pt) node[below] {$\delta_4$};

\draw[thick]
(0,0) -- (1,1.4) -- (2,2.5)  -- (4,3.8);

\node at (4.6,3.8) {$\cdots$};
\draw[dashed] (4,3.8) -- (5.2,3.8);

\node at (5.5,3.8) {$\delta_5,\delta_6,\ldots$};

\end{tikzpicture}
\small Case 4
\end{minipage}

\vspace{0.6cm}

\begin{minipage}{0.48\textwidth}
\centering
\begin{tikzpicture}[scale=0.8]

\draw[->] (0,0) -- (6,0) node[right] {$k$};
\draw[->] (0,0) -- (0,5) node[above] {$\delta_k$};

\fill (0,0)   circle (2pt) node[below] {$\delta_0$};
\fill[red] (1,1) circle (2pt) node[below] {$\delta_1$};
\fill (2,2.5) circle (2pt) node[below] {$\delta_2$};
\fill[red] (3,2.5) circle (2pt) node[below] {$\delta_3$};
\fill (4,3.8) circle (2pt) node[below] {$\delta_4$};

\draw[thick]
(0,0)  -- (2,2.5) -- (4,3.8);

\node at (4.6,3.8) {$\cdots$};
\draw[dashed] (4,3.8) -- (5.2,3.8);

\node at (5.5,3.8) {$\delta_5,\delta_6,\ldots$};

\end{tikzpicture}
\small Case 5
\end{minipage}
\caption{Possible configurations of the truncated Newton polygon determined by the coefficients $\delta_0,\ldots,\delta_4$. Black points are vertices of the Newton polygon (saturated coefficients), while red points lie strictly below the polygon (non-saturated coefficients). The dashed segment indicates that the Newton polygon continues beyond $\delta_4$.}
\label{fig:truncated_newton}

\end{figure}

\begin{figure}[H]
\centering

\begin{minipage}{0.48\textwidth}
\centering
\begin{tikzpicture}[scale=0.8]

\draw[->] (0,0) -- (6,0) node[right] {$k$};
\draw[->] (0,0) -- (0,5) node[above] {$\delta_k$};

\fill (0,0)   circle (2pt) node[below] {$\delta_0$};
\fill[red]  (1,0.7) circle (2pt) node[below] {$\delta_1$};
\fill[red]  (2,1.5) circle (2pt) node[below] {$\delta_2$};
\fill (3,3.3) circle (2pt) node[below] {$\delta_3$};
\fill (4,3.8) circle (2pt) node[below] {$\delta_4$};

\draw[thick]
(0,0)  -- (3,3.3) -- (4,3.8);

\node at (4.6,3.8) {$\cdots$};
\draw[dashed] (4,3.8) -- (5.2,3.8);

\node at (5.5,3.8) {$\delta_5,\delta_6,\ldots$};

\end{tikzpicture}
\small Case 1
\end{minipage}
\hfill
\begin{minipage}{0.48\textwidth}
\centering
\begin{tikzpicture}[scale=0.8]

\draw[->] (0,0) -- (6,0) node[right] {$k$};
\draw[->] (0,0) -- (0,5) node[above] {$\delta_k$};

\fill (0,0)   circle (2pt) node[below] {$\delta_0$};
\fill (1,1.4) circle (2pt) node[below] {$\delta_1$};
\fill[red] (2,2) circle (2pt) node[below] {$\delta_2$};
\fill[red] (3,2.5) circle (2pt) node[below] {$\delta_3$};
\fill (4,3.8) circle (2pt) node[below] {$\delta_4$};

\draw[thick]
(0,0) -- (1,1.4) -- (4,3.8);

\node at (4.6,3.8) {$\cdots$};
\draw[dashed] (4,3.8) -- (5.2,3.8);

\node at (5.5,3.8) {$\delta_5,\delta_6,\ldots$};

\end{tikzpicture}

\small Case 2
\end{minipage}

\caption{Configurations of the truncated Newton polygon determined by $\delta_0,\ldots,\delta_4$ that cannot occur. The red points indicate coefficients that would have to be non-saturated, leading to a violation of the concavity of the Newton polygon. The dashed segment indicates the continuation of the Newton polygon beyond $\delta_4$.}\label{fig:impossible}

\end{figure}
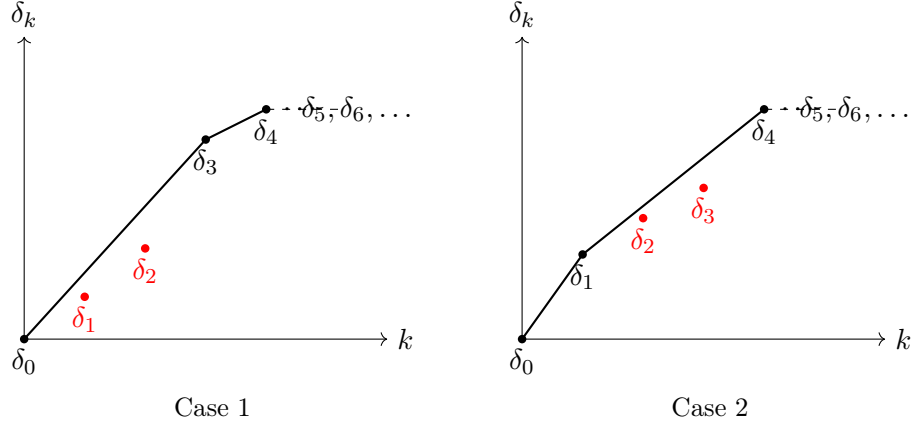

\begin{example}
Consider the sequence of coefficients
\[
(\delta_0,\delta_1,\ldots,\delta_9)
=(0,9,18,26,38,47,56,64,72,73).
\]
Then
\[
\delta_4+\delta_2=56>52=2\delta_3,
\]
and
\[
\delta_4+2\delta_1=56>54=3\delta_2.
\]
Hence, neither inequality of \Cref{prop} holds. Therefore, this sequence cannot be the coefficient sequence of the tropical characteristic polynomial of a symmetric matrix.

On the other hand, this sequence is realizable by a nonsymmetric matrix. Indeed, consider
\[
A=
\begin{bmatrix}
-4& 2& 3& 4&10&-1&-1& 3&-3\\
-3& 3& 4& 3& 8&-1& 1& 6& 0\\
-4& 1& 1&-6& 7&-1&-6&10&-6\\
-3&-4&-3& 9& 2& 1&-3&-2& 2\\
-6& 2&10& 6&-3& 5& 5& 1&-1\\
 4& 8&-5& 6& 0&-6&-1& 1&-4\\
-2& 8&-5&-5&-4& 8& 9& 6&-3\\
 3&-2&-4& 8&-6& 3&-4& 7& 9\\
 5&-3&-4& 9& 9& 8&10&-5& 5
\end{bmatrix}.
\]
The tropical algebraic eigenvalues of $A$ are
\[
1,\;8,\;9,\;9.5,
\]
with multiplicities
\[
1,\;2,\;2,\;4,
\]
respectively. The corresponding tropical characteristic polynomial has coefficients
\[
(\delta_0,\delta_1,\ldots,\delta_9)
=(0,9,18,26,38,47,56,64,72,73).
\]
Thus, the inequalities of \Cref{prop} provide a necessary condition for symmetric matrices, but they need not hold for nonsymmetric matrices.
\end{example}


       \section{Future work}\label{sec:6}

A natural direction for further research is to fully understand the concavity-type structure of the sequence $(\delta_m)$ associated with symmetric tropical matrices. Our results for the first coefficients of the tropical characteristic polynomial of symmetric matrices suggest a strong regularity phenomenon in the behavior of these coefficients, which appears to govern the geometry of the corresponding Newton polygons.
In particular, we expect that this pattern extends to all coefficients, leading to the following conjectural characterization.

\begin{conjecture}\label{conj_future}
For $3 \le m \le n$, at least one of the following inequalities holds:
\begin{enumerate}
\item $\delta_m - \delta_{m-1} \le \delta_{m-1} - \delta_{m-2}$, or
\item $\delta_m - \delta_{m-2} \le 2(\delta_{m-2} - \delta_{m-3})$.
\end{enumerate}
\end{conjecture}

If such concavity properties hold in general, they would impose strong constraints on the shape of the Newton polygon of the tropical characteristic polynomial of symmetric matrices. This would in turn provide a more precise description of tropical eigenvalues.
Moreover, it would open the way to a deeper understanding of eigenvalues of symmetric matrices over Puiseux series, in particular their asymptotic behavior, by linking algebraic structure with the geometry of the associated valuation data.







\bibliographystyle{alpha} 
\bibliography{refs} 
\end{document}